\documentclass[12pt,twoside, a4paper]{amsart}
\usepackage[utf8]{inputenc}
\usepackage{amsmath,amssymb,amsthm,mathtools, hyperref}
\usepackage{geometry, color}
\newtheorem{theorem}{Theorem}[section]
\newtheorem{lemma}[theorem]{Lemma}

\theoremstyle{definition}
\newtheorem{definition}[theorem]{Definition}
\newtheorem{remark}[theorem]{Remark}
\numberwithin{equation}{section}
\newcommand{\bC}{\mathbb{C}}
\newcommand{\bD}{\mathbb{D}}
\newcommand{\cK}{\mathcal{K}}
\newcommand{\cH}{\mathcal{H}}
\newcommand{\cF}{\mathcal{F}}
\newcommand{\cE}{\mathcal{E}}
\newcommand{\cG}{\mathcal{G}}

\newcommand{\cQ}{\mathcal{Q}}
\newcommand{\cD}{\mathcal{D}}
\newcommand{\bT}{\mathbb{T}}
\newcommand{\tcK}{{\widetilde \cK}}

\newcommand{\tcE}{{\widetilde {\mathcal E}}}
\newcommand{\tcF}{{\widetilde {\mathcal F}}}
\newcommand{\tcG}{{\widetilde {\mathcal G}}}

\newcommand{\tV}{{\widetilde V}}
\newcommand{\tY}{{\widetilde  Y}}
\newcommand{\tTheta}{{\widetilde \Theta}}
\newcommand{\tXi}{{\widetilde \Xi}}

\newcommand{\tW}{{\widetilde W}}

\newcommand{\sbm}[1]{\left[\begin{smallmatrix} #1
		\end{smallmatrix}\right]}

\title[Isometries satisfying automorphic covariance relation]{Functional Model for Isometric Pairs Satisfying an Automorphic Covariance Relation - I}
\author[Ghosh]{Sourav Ghosh}
\address{Department of Mathematics, Indian Institute of Science Education and Research Pune, Maharashtra 411008}
\email{sourav.ghosh@students.iiserpune.ac.in}
\author[Sau]{Haripada Sau}
\address{Department of Mathematics, Indian Institute of Science Education and Research Pune, Maharashtra 411008}
\email{hsau@iiserpune.ac.in}
\thanks{Sau was supported by the Anusandhan National Research Foundation (ANRF), Government of India, under the Mathematical Research Impact-Centric Support (MATRICS) scheme (Grant No. ANRF/ARGM/2025/000573/MTR)}

\begin{document}

\begin{abstract}
We derive a canonical functional model for a pair of isometric operators $(V_1, V_2)$ acting on a Hilbert space $\mathcal{H}$ that satisfy the automorphic covariance relation $V_1 V_2 = V_2 \varphi(V_1)$, where $\varphi$ is an automorphism of the unit disk $\mathbb{D}$. This extends the classical Bercovici--Douglas--Foias functional model for commuting isometric pairs to a dynamically twisted setting. When both $V_1$ and $V_2$ are pure isometries (shifts), the model simplifies to $(M_z, M_{\Theta} W_\varphi^*)$ acting on $H^2(\mathcal{F})$, where $W_\varphi$ is a unitary weighted composition operator on the Hardy space $H^2$ and $\Theta \in H^\infty(\mathbb{D}, \mathcal{B}(\mathcal{F}))$ is an operator-valued inner function such that $M_{\Theta}$ is a shift. As an application, we establish a complete unitary classification theorem for automorphically covariant isometric pairs by identifying a complete set of joint unitary invariants. Notably, our results are novel even in the special case of rotation automorphisms ($V_1 V_2 = q V_2 V_1$ with $|q| = 1$).
\end{abstract}
\maketitle

\section{Introduction}
The structural analysis of isometric operators on Hilbert spaces originates in the foundational work of Wold \cite{Wold1954}, who established that any single isometry decomposes uniquely into a direct sum of a unitary operator and a \textit{shift} (also called a \textit{pure isometry}), i.e., an isometry $V$ such that $V^{*n}\to 0$ strongly as $n\to\infty$. Extending this structural paradigm to systems of multiple operators has been a central theme in multivariable operator theory and function theory.

For commuting isometric pairs $(V_1, V_2)$, the first functional model was introduced by Berger, Coburn, and Lebow \cite{BCL1978} by embedding the pair into the Wold decomposition of the product isometry $V = V_1 V_2$. Subsequent structural investigation by Ga\c{s}par and Ga\c{s}par \cite{GG2004} and Popovici \cite{Popovici2004} further clarified multivariable Wold-type decompositions for bi-isometries and commuting isometric pairs. A major conceptual advance was achieved in a series of papers by Bercovici, Douglas, and Foias \cite{BDF2006, BDF2010, BDF2012}, who constructed canonical functional models (the BDF models) for bi-isometries and multi-isometries by embedding the pair directly into the functional-model Wold decomposition of a single component isometry $V_1$, rather than the product $V_1 V_2$. Recently, Ball and the second author of the present manuscript developed in \cite{BS2026}  a unified theory of dilations and canonical models for commuting contraction pairs, establishing complete unitary classification theorems via a set of parameters arising from the model theory of  Berger, Coburn, and Lebow \cite{BCL1978}.

While the model theory for commuting isometric pairs, governed by the trivial dynamics $\varphi(z) = z$, is well established, non-commutative operator pairs arise naturally when shift operators interact with conformal domain dynamics. A central non-commutative structure is governed by the \emph{automorphic covariance relation}
\begin{equation} \label{eq:intro_covar}
V_1 V_2 = V_2 \varphi(V_1),
\end{equation}
where $\varphi \in \operatorname{Aut}(\mathbb{D})$.
Pairs satisfying \eqref{eq:intro_covar} provide the canonical operator-theoretic framework for representations of dynamical systems $(\mathbb{D}, \varphi)$. Specifically, covariant pairs $(V_1, V_2)$ parametrize the isometric representations of the semicrossed product algebra $\mathcal{A}(\mathbb{D}) \rtimes_\varphi \mathbb{Z}_{\ge 0}$, where $\mathcal{A}(\mathbb{D})$ is the Banach algebra of holomorphic self-maps of $\mathbb{D}$ that extends continuously over the unit circle $\mathbb{T}$. The covariance dilation theorems allow to extend this representation to an isometric representation  of the crossed product algebra $C(\mathbb{T}) \rtimes_\varphi \mathbb{Z}_{\ge 0}$ (see e.g., \cite{BP1998, DK2010, DK2012, GS2026}). This representation simultaneously generates the universal $C^*$-algebras associated with Hardy shifts twisted by automorphic composition operators ($C_\varphi:f\mapsto f\circ\varphi$). Building a model theory for \eqref{eq:intro_covar} thus extends classical multivariable spectral theory from static function theory to dynamically twisted non-commutative domains.

Constructing a canonical BDF-type model and proving a complete unitary classification theorem for automorphically covariant isometric pairs presents substantial new hurdles that do not occur in the classical commuting setting. For instance,

\begin{enumerate}
    \item In the BDF model for commuting pairs, embedding $V_2$ into the Wold decomposition of $V_1$ yields a block entry $A$ commuting with the shift $M_z^\cF$, which by the classical theory of von Neumann algebra forces $A$ to be $M_\Theta$ where $\Theta$ is an analytic multiplier. Under the covariance relation \eqref{eq:intro_covar}, $A$ satisfies $M_z^\cF A = A M_\varphi^\cF$, forcing $A$ to be of the twisted form $M_\Theta C_{\varphi^{-1}}$ to couple the multiplier $M_\Theta$ with the conformal composition operator $C_{\varphi^{-1}}$. This structural result is Lemma \ref{L:WightCompos} below.

    \item The defect operator associated with the second component $V_2$ is governed by the operator $ I - C_{\varphi^{-1}}^* M_\Theta^* M_\Theta C_{\varphi^{-1}}$. Due to the conformal change of variables induced by $C_{\varphi^{-1}}$, the computation of the adjoint picks up the conformal derivative factor $|\varphi'(\zeta)|$ on $\mathbb{T}$. As a result, the defect factorization requires analyzing the \emph{weighted defect operator}
    $$
    \Delta_{\Theta, \varphi}(\zeta) := \left( I_\cF - |\varphi'(\zeta)| \Theta(\varphi(\zeta))^* \Theta(\varphi(\zeta)) \right)^{\frac{1}{2}} \quad \mbox{for }\zeta \in \mathbb{T}.
    $$


\end{enumerate}

The primary aim of this paper is to overcome these challenges by deriving a self-contained functional model and proving a complete unitary classification theorem for automorphically covariant isometric pairs. Notably, our results are novel even in the special case of the $q$-twisted relation $V_1V_2=qV_2V_1$, where $q$ is a unimodular complex number.

\section{The Main Results}
This section is divided into three subsections. The first establishes several preliminary results on isometric/unitary pairs satisfying automorphic covariance relations; these lemmas are used later and may be of independent interest. The remaining two subsections present the primary functional model and classification theorems.
\subsection{A few key lemmas}
Each of the following lemmas will be used in the derivation of the model and the unitary invariance.
\begin{lemma}\label{L:WoldReduces}
Let $V$ be an isometry and $W$ be a unitary acting on a Hilbert space $\mathcal{H}$ satisfying the covariance relation
$$
W V = V f(W)
$$
for some finite Blaschke function $f(z)=\prod_{j=1}^n e^{i\theta}(z-a_j)(1-\overline{a_j}z)^{-1}$. Then the Wold decomposition of $V$ reduces $W$ into the direct sum of two unitary operators. 
\end{lemma}
\begin{proof}
If follows by an inductive argument that
$$
W V^n = V^n (f^{\circ n}(W)).
$$
where $f^{\circ n}$ denotes the $n$-fold composition of $f$. Since $f$ is a finite Blaschke function, $f|_{\mathbb{T}}:\mathbb{T}\to\mathbb{T}$ is a continuous function, it therefore follows by the spectral theorem that $f(W)$ is a unitary operator, whenever $W$ is so. This observation together with the identity above yields
$$
W(\operatorname{Ran}V^n) = V^n f^{\circ n}(W)\mathcal{H} = V^n\mathcal{H} = \operatorname{Ran}V^n.
$$
Thus, each subspace $\mathrm{Ran}(V^n)$ reduces $W$. Since the unitary part of $V$ in its Wold decomposition is defined on $\mathcal{H}_u(V) = \bigcap_{n=0}^\infty \mathrm{Ran}(V^n)$, it follows that $\mathcal{H}_u(V)$ is reducing for $W$. 
\end{proof}

\begin{lemma}\label{L:UniShiftCov}
Let $W$ be a unitary operator on $H^2(\mathcal{F})$, and let $\varphi: \mathbb{D} \to \mathbb{D}$ be an analytic automorphism. If $W$ satisfies the covariance relation
\begin{align}\label{cov}
W M_z = M_z \varphi(W),
\end{align}
then there is a unitary operator $u$ on the coefficient space $\cF$ so that $W$ decomposes as the direct sum
$$
W = \bigoplus_{n=0}^\infty \varphi^{\circ n}(u).
$$
\end{lemma}
\begin{proof}
Since $W$ is a unitary and $\varphi$ is an automorphism, $\varphi(W)$ is a unitary too and therefore it is trivial to see from \eqref{cov} that
$W$ maps $\operatorname{Ran}(M_z)$ surjectively onto $\operatorname{Ran}(M_z)$. Thus $\operatorname{Ran}(M_z )$ is reducing for $W$. Let us denote
$$
u := W|_{\mathcal{F}}.
$$ Then $u$ is a unitary operator on $\mathcal{F}$.

For any vector $x \in \mathcal{F} \subset H^2(\mathcal{F})$, apply \eqref{cov} to see that
$$
W(z x) = W M_z x = M_z \varphi(W) x = M_z (\varphi(u) x) = z \varphi(u) x.
$$
This shows that $W(z \mathcal{F}) = z \mathcal{F}$ and $W|_{z \mathcal{F}} = \varphi(u)$.

We use mathematical induction to conclude the proof. Suppose $W|_{z^k \mathcal{F}} = \varphi^{\circ k}(u)$ for some $k \geq 0$. By functional calculus, 
$$
\varphi(W)|_{z^k \mathcal{F}} = \varphi\left(\varphi^k(u)\right) = \varphi^{k+1}(u).
$$ 
Thus for any $x \in \mathcal{F}$, we apply \eqref{cov} to see that
$$
W(z^{k+1} x) = W M_z (z^k x) = M_z \varphi(W) (z^k x) = M_z \left( z^k \varphi^{k+1}(u) x \right) = z^{k+1} \varphi^{k+1}(u) x.
$$
Hence $W|_{z^{k+1} \mathcal{F}} = \varphi^{k+1}(u)$, completing the induction. 
\end{proof}

\begin{lemma}\label{L:WightCompos}
Let $\varphi$ be an automorphism and $A$ be a bounded operator on $H^2(\mathcal{E})$ such that 
$$
\varphi(M_z)A=AM_z.
$$Then $A=M_{\Theta}C_\varphi$ for some $\Theta$ in $H^{\infty}(\mathcal{B(E)})$ with $\|\Theta\|_\infty \leq \|A\|\sqrt{\|\varphi'\|_{\infty,\bT}}$.
\end{lemma}
\begin{proof}
It is easy to see that $\varphi(M_z)M_{\Theta}C_\varphi=M_{\Theta}C_\varphi M_z$ for any $\Theta \in H^{\infty}(\mathcal{B(E)})$. To prove the converse, observe that  $\varphi(M_z)A=AM_z$ implies $T_{p\circ \varphi}A=AT_p$ for all $p\in \bC[z]$. Define $\Theta:\bD\to \mathcal{B(E)}$ by 
$$
\Theta(z)e=A(1\otimes e)(z).
$$
Hence for every scalar-valued polynomial $p$,
\begin{align*}
    A(p\otimes e)(z)=p(\varphi(z))\Theta(z)e .
\end{align*}By density of $\mathcal{E}$-valued polynomials in $H^2(\mathcal{E})$ and boundedness of $A$, the same identity extends to all $g\in H^2(\mathcal{E})$:
\begin{align}\label{mult-compo}
Ag(z)=\Theta(z)g( \varphi(z))
\quad\mbox{for all } z\in\bD.
\end{align}
By the reproducing property of the kernel $k\otimes I_\cE$, we have 
$$
\langle \Theta(z)e_1,e_2\rangle=\langle A(1\otimes e_1)(z),e_2\rangle=\langle A(1\otimes e_1),k_z\otimes e_2\rangle.
$$ This implies $\Theta(z)$ is bounded for all $z\in \bD$, and $\Theta$ is holomorphic. Once we prove that $\Theta$ is bounded, the conclusion follows immediately. For all $z,w\in\bD$, we compute 
\begin{align*}
k_z(\varphi(w))\langle\Theta(w)e_1,e_2 \rangle _\cE &= k_z(\varphi(w))\langle e_1,\Theta(w)^*e_2 \rangle _\cE\\
&=\left\langle k_z(\varphi(w)) e_1,\Theta(w)^*e_2\right\rangle_\cE\\
&=\left\langle \Theta(w) k_z(\varphi(w)) e_1,e_2\right\rangle_\cE\\
&=\langle A(k_z\otimes e_1)(w),e_2\rangle_{\cE}=\langle A(k_z\otimes e_1),k_w\otimes e_2\rangle_{H^2(\cE)}.    
\end{align*}
In the computation above, we used \eqref{mult-compo} for $g=k_z\otimes e_1$. We now take the modulus both sides to get 
\begin{align*}
    |k_z(\varphi(w))||\langle \Theta(w)(e_1),e_2\rangle|=|\langle A(k_z\otimes e_1),k_w\otimes e_2\rangle|\leq \|A\|\|k_z\|\|k_w\|\|e_1\|\|e_2\|.
\end{align*} Putting $z=\overline{\varphi(w)}$ yields $$|\langle\Theta(w)e_1,e_2\rangle|\leq \|A\| \sqrt{\dfrac{1-|\varphi(w)|^2}{1-|w|^2}}\|e_1\|\|e_2\|=\|A\|\sqrt{|\varphi'(w)|}\|e_1\|\|e_2\|.
$$It then implies that for all $w\in\bD$, 
$$
\|\Theta(w)\|=\sup_{\|e_1\|=\|e_2\|=1}|\langle \Theta(w)e_1,e_2\rangle|\leq \|A\|\sqrt{\|\varphi'\|_\infty}.
$$ Hence we conclude from \eqref{mult-compo} that $A=M_\Theta C_\varphi$.
\end{proof}
\begin{remark}
 The above lemma is true for a larger class of functions than the automorphism. In \cite[Theorem~2.3]{GS2026}, we showed that Lemma~\ref{L:WightCompos} holds for finite Blaschke functions as well.
\end{remark}
\begin{lemma}\label{lem:H2density}
For any $\varphi \in \operatorname{Aut}(\mathbb{D})$ and Hilbert space $\cF$,
$$
\overline{\bigcup_{n=0}^\infty \varphi^{-n} H^2(\cF)} = L^2(\cF).
$$
\end{lemma}
\begin{proof}
It suffices to establish the scalar case $L^2$. Since $\varphi \in \operatorname{Aut}(\mathbb{D})$, the composition operator $C_\varphi f = f \circ \varphi$ is a bounded automorphism on $H^2$. Since $H^2 = \overline{\operatorname{span}}\{z^k : k \geq 0\}$, we have $H^2 = \overline{\operatorname{span}}\{\varphi^k : k \geq 0\}$.
Taking unions over $n \geq 0$ yields:
$$
\bigcup_{n=0}^\infty \varphi^{-n} H^2 \supseteq \operatorname{span}\{\varphi^m : m \in \mathbb{Z}\}.
$$
Because $\varphi \colon \mathbb{T} \to \mathbb{T}$ is a homeomorphism, $\operatorname{span}\{\varphi^m : m \in \mathbb{Z}\}$ is a point-separating subalgebra of $C(\mathbb{T})$ closed under complex conjugation. By Stone--Weierstrass, it is dense in $C(\mathbb{T})$ and thus dense in $L^2$. Tensoring with $\cF$ yields $L^2(\cF)$.
\end{proof}

Finally, the we recall the following lemma from \cite[Lemma 2.3.1]{BS2026}.
\begin{lemma}\label{L:AuxLemma}
\begin{enumerate}
\item
If
$\cK$ and $\cK'$ are Hilbert spaces, $U$ is a unitary operator on $\cK$, $V$ is a pure isometry on $\cK'$, and $\Gamma \colon \cK \to \cK'$
is a bounded linear operator such that $\Gamma U = V \Gamma$, then $\Gamma = 0$.

\item Suppose that $V$ and $V'$ are pure isometries on $\cK_1$ and $\cK_1'$ respectively, $U$ and $U'$ are unitary operators on
$\cK_2$ and $\cK'_2$ respectively, and $\Gamma = \sbm{ \Gamma_{11} & \Gamma_{12} \\ \Gamma_{21} & \omega }$
is a unitary operator from $\cK_1 \oplus \cK_2$ to $\cK'_1 \oplus \cK'_2$ which intertwines $V \oplus U$ with $V' \oplus U'$:
\begin{equation}   \label{block-intertwine}
\begin{bmatrix} \Gamma_{11}  & \Gamma_{12}  \\ \Gamma_{21} & \omega \end{bmatrix} \begin{bmatrix} V & 0 \\ 0 & U \end{bmatrix}
= \begin{bmatrix} V' & 0 \\ 0 & U' \end{bmatrix} \begin{bmatrix} \Gamma_{11}  & \Gamma_{12}  \\ \Gamma_{21} & \omega \end{bmatrix}.
\end{equation}
Then $\Gamma$ is block-diagonal, i.e.,  $\Gamma_{12} = 0$ and $\Gamma_{21} = 0$.
\end{enumerate}
\end{lemma}

\subsection{The model construction}\label{SS:FunctModel}

Let $(V_1, V_2)$ be an isometric operator pair on a Hilbert space $\cH$ satisfying the automorphic covariance relation
\begin{equation} \label{AutomorphicCov}
V_1 V_2 = V_2 \varphi(V_1),
\end{equation}
where $\varphi \in \operatorname{Aut}(\mathbb{D})$. Our goal is to obtain a model for $(V_1, V_2)$. This will be achieved in several intermediate unitary identification. For notational simplicity, we do not make distinction in notation of the original pair $(V_1,V_2)$ and its intermediate unitary conjugations.

For a contractive operator $T$, the following notations for defect operator and defect space will be used throughout the paper:
$$
D_T=(I-T^*T)^{1/2} \quad \mbox{and}\quad
\cD_T=\operatorname{Ran}D_T.
$$For an isometric operator $V$, let $Q_{V^*}$ denote the orthogonal projection
$$
Q_{V^*} = \operatorname{SOT-}\lim V^nV^{*n}
$$and $\cQ_{V_1^*}$ denote its range.

Let
$
\cF := \cD_{V_1^*}$ denote the defect space of $V_1^*$. By the Wold decomposition theorem, the isometry $V_1$ is unitarily equivalent to
$$
 \begin{bmatrix} M_z^\cF & 0 \\ 0 & W_1 \end{bmatrix}:\begin{bmatrix} H^2(\cF) \\ \cQ_{V_1^*} \end{bmatrix}\to \begin{bmatrix} H^2(\cF) \\ \cQ_{V_1^*} \end{bmatrix},
$$
where $W_1$ is a unitary operator on $\cQ_{V_1^*}=\operatorname{Ran}Q_{V_1^*}$. Representing $V_2$ on $\begin{bmatrix} H^2(\cF) \\ \cQ_{V_1^*} \end{bmatrix}$ as $\begin{bmatrix} A & B \\ C & E \end{bmatrix}$, condition \eqref{AutomorphicCov} becomes
$$
\begin{bmatrix} M_z^\cF & 0 \\ 0 & W_1 \end{bmatrix} \begin{bmatrix} A & B \\ C & E \end{bmatrix} = \begin{bmatrix} A & B \\ C & E \end{bmatrix} \begin{bmatrix} M_\varphi^\cF & 0 \\ 0 & \varphi(W_1) \end{bmatrix}.
$$
This yields:
\begin{enumerate}
    \item[(M1)] $M_z^\cF A = A M_\varphi^\cF$, or equivalently, $\varphi^{-1}(M_z^\cF) A = A M_z^\cF$ since $\varphi$ is an automorphism. Applying Lemma \ref{L:WightCompos} with $\varphi^{-1}$ gives $A = M_\Theta C_{\varphi^{-1}}$ for a bounded analytic function $\Theta$ on $\mathbb{D}$ with values in $\mathcal{B}(\cF)$.
    \item[(M2)] $M_z^\cF B = B \varphi(W_1)$. Since $W_1$ is unitary and $M_z^\cF$ is a pure shift operator, by Part (1) of Lemma \ref{L:AuxLemma}, $B=0$.
    \item[(M3)] $W_1 C = C M_\varphi^\cF$ and $W_1 E = E \varphi(W_1)$.
\end{enumerate}
All this amounts to the representation 
$$
\begin{bmatrix} M_\Theta C_{\varphi^{-1}} & 0 \\ C & E \end{bmatrix}:\begin{bmatrix} H^2(\cF) \\ \cQ_{V_1^*} \end{bmatrix}\to \begin{bmatrix} H^2(\cF) \\ \cQ_{V_1^*} \end{bmatrix}
$$ of the isometry $V_2$ on the Wold decomposition space of $V_1$. Since $V_2$ is an isometry, the above block matrix representation readily implies the following:
\begin{enumerate}
    \item[(I1)] $E$ is an isometry acting on $\cQ_{V_1^*}$.
    \item[(I2)] $\operatorname{Ran} C \subset \cD_{E^*} = (\operatorname{Ran} E)^\perp$.
    \item[(I3)] $C^* C = I_{H^2(\cF)} - C_{\varphi^{-1}}^* M_\Theta^* M_\Theta C_{\varphi^{-1}}$.
\end{enumerate}
The Wold decomposition applied to the isometry $E$ splits the space $\cQ_{V_1^*}$ as
$$
\cQ_{V_1^*} = \begin{bmatrix} H^2(\cD_{E^*}) \\ \cQ_{E^*} \end{bmatrix}
$$ and the isometry $E$ as
$$
E=\begin{bmatrix}
    M_z^{\cD_{E^*}} & 0\\
    0 & Y_2
\end{bmatrix}:
\begin{bmatrix}
     H^2(\cD_{E^*}) \\ \cQ_{E^*} 
\end{bmatrix} \to \begin{bmatrix}
     H^2(\cD_{E^*}) \\ \cQ_{E^*} 
\end{bmatrix}.
$$Under this decomposition, we must have
$
C = \begin{bmatrix} C_1 \\ 0 \end{bmatrix}
$ where $C_1 \colon H^2(\cF) \to \cD_{E^*}$ satisfies
\begin{align}\label{IdenC1}
C_1^* C_1 = I_{H^2(\cF)} - C_{\varphi^{-1}}^* M_\Theta^* M_\Theta C_{\varphi^{-1}}.
\end{align}
Apply Lemma \ref{L:WoldReduces} to the unitary/isometry pair $(W_1,E)$ to get $W_1$ in the direct sum form
$$
W_1=\begin{bmatrix}
    W_1^1 & 0\\
    0 & Y_1
\end{bmatrix}:
\begin{bmatrix} H^2(\cD_{E^*}) \\ \cQ_{E^*} \end{bmatrix}\to \begin{bmatrix} H^2(\cD_{E^*}) \\ \cQ_{E^*} \end{bmatrix}.
$$Furthermore, application of Lemma \ref{L:UniShiftCov} to the unitary/shift pair $(W_1^1, M_z^{\cD_{E^*}})$ yields a unitary operator $u$ on $\cD_{E^*}$ so that
$$
W_1^1 = I_{H^2}\otimes_\varphi u := \bigoplus_{n\geq 0} \varphi^{\circ n}(u).
$$
All this unitary identification transforms our original isometric pair $(V_1, V_2)$ to 
\[
\left(
\begin{bmatrix} M_z^\cF & 0 & 0 \\ 0 & I_{H^2} \otimes_\varphi u & 0 \\ 0 & 0 & Y_1 \end{bmatrix}, 
 \begin{bmatrix} M_\Theta C_{\varphi^{-1}} & 0 & 0 \\ C_1 & M_z^{\cD_{E^*}} & 0 \\ 0 & 0 & Y_2 \end{bmatrix}\right)
:\begin{bmatrix} H^2(\cF) \\ H^2(\cD_{E^*}) \\ \cQ_{E^*} \end{bmatrix}\to\begin{bmatrix} H^2(\cF) \\ H^2(\cD_{E^*}) \\ \cQ_{E^*} \end{bmatrix}
\]
where $u$ is a unitary operator on $\cD_{E^*}$, $(Y_1, Y_2)$ a unitary pair on $\cQ_{E^*}$ satisfying 
$$
Y_1 Y_2 = Y_2 \varphi(Y_1),
$$and $C_1$ satisfies the intertwining relation 
\begin{align}\label{Intu1}
(I_{H^2}\otimes_\varphi u) C_1 = C_1 M_\varphi^\cF.
\end{align}

We shall resume this construction of the model after we compute $C_{\varphi^{-1}}^* M_\Theta^* M_\Theta C_{\varphi^{-1}}$ on $L^2(\cF)$.
For $f, g \in L^2(\cF)$,
$$
\langle C_{\varphi^{-1}} f, g \rangle_{L^2} = \int_\mathbb{T} \langle f(\varphi^{-1}(w)), g(w) \rangle_\cF \, d\sigma(w).
$$
Applying change of variables $\zeta = \varphi^{-1}(w)$ or equivalently $w = \varphi(\zeta)$, we get
$$
\langle C_{\varphi^{-1}} f, g \rangle_{L^2} = \int_\mathbb{T} \langle f(\zeta), g(\varphi(\zeta)) \rangle_\cF \, |\varphi'(\zeta)| d\sigma(\zeta) = \langle f, C_{\varphi^{-1}}^* g \rangle_{L^2}.
$$
Thus, the adjoint $C_{\varphi^{-1}}^*$ acts on $g \in L^2(\cF)$ as:
$$
(C_{\varphi^{-1}}^* g)(\zeta) = |\varphi'(\zeta)| \, g(\varphi(\zeta)).
$$
We now compute $C_{\varphi^{-1}}^* M_\Theta^* M_\Theta C_{\varphi^{-1}} f$:
\begin{align*}
\left( C_{\varphi^{-1}}^* M_\Theta^* M_\Theta C_{\varphi^{-1}} f \right)(\zeta) &= |\varphi'(\zeta)| \left( M_\Theta^* M_\Theta C_{\varphi^{-1}} f \right)(\varphi(\zeta)) \\
&= |\varphi'(\zeta)| \, \Theta(\varphi(\zeta))^* \Theta(\varphi(\zeta)) f(\zeta).
\end{align*}
Therefore
$$
I_{L^2(\cF)} - C_{\varphi^{-1}}^* M_\Theta^* M_\Theta C_{\varphi^{-1}} = M_{\Delta_{\Theta, \varphi}}^2,
$$
where 
$$
\Delta_{\Theta, \varphi}(\zeta) = \left( I_\cF - |\varphi'(\zeta)| \Theta(\varphi(\zeta))^* \Theta(\varphi(\zeta)) \right)^{1/2}.
$$
In view of \eqref{IdenC1}, we get
$$
C_1^* C_1 = P_{H^2(\cF)} \Delta_{\Theta, \varphi}^* \Delta_{\Theta, \varphi} |_{H^2(\cF)}
$$
giving the expression of the operator $C_1$  
$$
C_1 = \iota \Delta_{\Theta, \varphi}|_{H^2(\cF)}
$$ for some isometry $\iota \colon \overline{\Delta_{\Theta, \varphi} H^2(\cF)} \to \cD_{E^*}$. Since the isometry $\iota$ takes values in $\cD_{E^*}$, which is identified as the space of constant functions in $H^2(\cD_{E^*})$, the intertwining condition \eqref{Intu1} translates to 
\begin{align}\label{Intu1'}
 u \iota = \iota M_\varphi \quad \text{on } \overline{\Delta_{\Theta, \varphi} H^2(\cF)}.
\end{align}

We now extend the domain of $\iota$. We need the following density result.
\begin{lemma}\label{lem:Ddense}
The subspace $\cD_{\text{dense}} := \bigcup_{n=0}^\infty \varphi^{-n} \Delta_{\Theta, \varphi} H^2(\cF)$ is dense in $\mathcal{M} := \overline{\Delta_{\Theta, \varphi} L^2(\cF)}$.
\end{lemma}
\begin{proof}
Since $\Delta_{\Theta, \varphi}(w)$ acts on coefficients while $\varphi(w)^{-n}$ is a scalar, multiplication by $\varphi^{-n}$ commutes with $\Delta_{\Theta, \varphi}$. Thus $\cD_{\text{dense}} \subset \mathcal{M}$.

Suppose $g \in \mathcal{M}$ satisfies $g \perp \cD_{\text{dense}}$. For all $n \geq 0$ and $f \in H^2(\cF)$,
$$
0 = \langle g, \varphi^{-n} \Delta_{\Theta, \varphi} f \rangle_{L^2(\cF)} = \langle g, \Delta_{\Theta, \varphi} (\varphi^{-n} f) \rangle_{L^2(\cF)}.
$$
Since $\Delta_{\Theta, \varphi}$ is a self-adjoint operator on $L^2(\cF)$, we obtain
$$
\langle \Delta_{\Theta, \varphi} g, \varphi^{-n} f \rangle_{L^2(\cF)} = 0.
$$
By Lemma~\ref{lem:H2density}, $\Delta_{\Theta, \varphi} g \perp L^2(\cF)$, forcing $\Delta_{\Theta, \varphi} g = 0$, so $g \in \operatorname{Ker}(\Delta_{\Theta, \varphi})$. Since $\mathcal{M} = \overline{\operatorname{Ran}(\Delta_{\Theta, \varphi})} = (\operatorname{Ker}(\Delta_{\Theta, \varphi}))^\perp$, we have $g \in \operatorname{Ker}(\Delta_{\Theta, \varphi}) \cap (\operatorname{Ker}(\Delta_{\Theta, \varphi}))^\perp = \{0\}$.
\end{proof}

We now come back to the model construction.
By Lemma~\ref{lem:Ddense}, we extend $\iota$ to an isometry $\hat{\iota} \colon \overline{\Delta_{\Theta, \varphi} L^2(\cF)} \to \cD_{E^*}$ defined on $\cD_{\text{dense}}$ by
$$
\hat{\iota}(\varphi^{-n} \Delta_{\Theta, \varphi} f) = u^{*n} \iota(\Delta_{\Theta, \varphi} f) \quad \mbox{where } f \in H^2(\cF) \mbox{ and } n \geq 0.
$$
It is evident that the extended isometry $\hat\iota:\overline{\Delta_{\Theta,\varphi}L^2(\cF)}\to\cD_{E^*}$  satisfies the intertwining $\hat{\iota} M_{\varphi^{-1}} = u^* \hat{\iota}$.
With the help of \eqref{Intu1'}, it is elementary to show that this extended isometry also satisfies the intertwining $\hat{\iota} M_\varphi = u \hat{\iota}$. Consequently, $\operatorname{Ran}\hat\iota\subset \cD_{E^*}$ is reducing for the unitary operator $u$ and $\hat{\iota}:\overline{\Delta_{\Theta,\varphi}L^2(\cF)}\to \operatorname{Ran}\hat\iota$ unitarily intertwines $M_\varphi|_{\overline{\Delta_{\Theta, \varphi} L^2(\cF)}}$ with $u|_{\operatorname{Ran} \hat{\iota}}$.
Thus with
$$
\cE_0 = \cD_{E^*} \ominus \operatorname{Ran} \hat{\iota}
$$we have the decomposition of $u:\cD_{E^*}\to \cD_{E^*}$:
$$
u = \begin{bmatrix} M_\varphi|_{\overline{\Delta_{\Theta, \varphi} L^2(\cF)}} & 0 \\ 0 & u_{\cE_0} 
\end{bmatrix}:
\begin{bmatrix}
    \overline{\Delta_{\Theta, \varphi} L^2(\cF)} \\
    \cE_0
\end{bmatrix}\to \begin{bmatrix}
    \overline{\Delta_{\Theta, \varphi} L^2(\cF)} \\
    \cE_0
\end{bmatrix}.
$$
Putting all this together we see that the original isometric pair is jointly unitarily equivalent to 
\begin{align}\label{ModelPair}
\left( \begin{bmatrix} 
M_z^\cF & 0 & 0 & 0 \\ 
0 & I_{H^2} \otimes_\varphi M_\varphi & 0 & 0 \\ 
0 & 0 & I_{H^2} \otimes_\varphi u_{\cE_0} & 0 \\ 
0 & 0 & 0 & Y_1 
\end{bmatrix}, 
\begin{bmatrix} 
M_\Theta C_{\varphi^{-1}} & 0 & 0 & 0 \\ 
\Delta_{\Theta, \varphi} |_{H^2(\cF)} & M_z^{\overline{\Delta_{\Theta, \varphi} L^2(\cF)}} & 0 & 0 \\ 
0 & 0 & M_z^{\cE_0} & 0 \\ 
0 & 0 & 0 & Y_2 
\end{bmatrix}\right)
\end{align}acting on the space
\begin{align}\label{ModelSpace}
 \begin{bmatrix}
 H^2(\cF) \\ H^2(\overline{\Delta_{\Theta, \varphi} L^2(\cF)}) \\ H^2(\cE_0) \\ \cQ_{E^*} 
\end{bmatrix}.
\end{align}

We leave it to the reader to check that if we have a tuple 
\begin{align}\label{Parameter}
\boldsymbol{\Pi}:=\{\cF,\cE_0, \Theta, u_{\cE_0}, Y_1,Y_2\}
\end{align} consisting of 
\begin{enumerate}
    \item[(P1)] Hilbert spaces $\cF,\cE_0$, 
    \item[(P2)] a bounded analytic function $\Theta:\mathbb D\to\mathcal B(\cF)$ such that $M_\Theta C_{\varphi^{-1}}$ is contractive for an automorphism $\varphi:\mathbb D\to\mathbb D$,
    \item[(P3)] a unitary operator $u_{\mathcal{E}_0}$ on $\mathcal{E}_0$, and
    \item[(P4)] a unitary pair $(Y_1,Y_2)$ such that $Y_1Y_2=Y_2\varphi(Y_1)$,
\end{enumerate}
 then the operator pair $(V_1,V_2)$ as in \eqref{ModelPair} is an isometric pair satisfying $V_1V_2=V_2\varphi(V_1)$.
 \begin{definition}
    Any set as in \eqref{Parameter} consisting of Hilbert spaces $\cF,\cE_0$, a bounded analytic function $\Theta:\mathbb D\to\mathcal B(\cF)$, a unitary operator $u_{\mathcal{E}_0}$ on $\mathcal{E}_0$, and a unitary pair $(Y_1,Y_2)$ as described in items (P1)--(P4) above will be referred to as a {\em parameter set}.
 \end{definition}
 Combining this terminology with the analysis leading to the model operator pair in \eqref{ModelPair}, we arrive at the following result.
\begin{theorem}\label{T:Model}
If $(V_1,V_2)$ is an isometric pair such that 
\begin{align}\label{Covariance}
    V_1V_2=V_2\varphi(V_1)
\end{align} for an automorphism $\varphi$ of the disk $\mathbb{D}$, then there is a parameter set $\{\cF,\cE_0, \Theta, u_{\cE_0}, Y_1,Y_2\}$ such that $(V_1,V_2)$ is jointly unitarily equivalent to the pair as in \eqref{ModelPair}. Conversely, given a parameter set $\{\cF,\cE_0, \Theta, u_{\cE_0}, Y_1,Y_2\}$, the operator pair as in \eqref{ModelPair} is an isometric pair satisfying the covariance relation \eqref{Covariance}.
\end{theorem}

\subsection{Model for bi-pure isometries satisfying an automorphic covariance relation}
The special case when both the isometries $V_1,V_2$ are shift operators is a case of interest. In the commuting case, the model pair is $(M_z,M_\Theta)$ on $H^2(\cF)$ where $\Theta$ is $\mathcal B(\cF)$-inner function such that $\Theta(0)^{*n}\to0$ as $n\to\infty$. Here we explore an analogue of this. We begin with the following result.
\begin{lemma}\label{L:WeightCompos}
    Let $\varphi \in \mathrm{Aut}(\bD)$ and $W_\varphi$ be the weighted composition operator on $H^2$ defined by
    $$
    (W_\varphi h)(z) = M_{\sqrt{\varphi'}}C_\varphi h(z)= \sqrt{\varphi'(z)} \, h(\varphi(z)).
    $$Then $W_\varphi$ is a unitary operator with $W_\varphi^*=W_{\varphi^{-1}}$.
\end{lemma}
\begin{proof}
Note that since $\varphi$ is an automorphism, $\varphi'$ is a never-vanishing holomorphic function on $\bD$ and hence $\sqrt{\varphi'}$ is holomorphic. What is actually true is that $\varphi'$ is in $H^\infty$. This can be seen by direct computation as follows. Suppose $\varphi(z) = e^{i\theta} (z - a)(1 - \overline{a}z)^{-1}$ where $a \in \bD$ and $\theta \in \mathbb{R}$. Direct computation gives $\varphi'(z) = e^{i\theta} (1 - |a|^2)(1 - \overline{a}z)^{-2}$. Since $|1 - \overline{a}z|\geq 1-|a| $ for all $z \in \overline{\bD}$, it follows that $\varphi'\in H^\infty$. Thus $W_\varphi$ is well defined on $H^2$.

To see that $\|W_\varphi h\|=\|h\|$ for every $h\in H^2$, we compute
\begin{align*}
\|W_\varphi h\|_{H^2}^2 = \frac{1}{2\pi} \int_{\bT} \left| \sqrt{\varphi'(\zeta)} \, h(\varphi(\zeta)) \right|^2 |d\zeta| = \frac{1}{2\pi} \int_{\bT} |h(\varphi(\zeta))|^2 |\varphi'(\zeta)| |d\zeta|
\end{align*}
and apply the change of variable $\eta=\varphi(\zeta)$ to note that
\begin{align*}
    \|W_\varphi h\|_{H^2}^2 = \frac{1}{2\pi} \int_{\bT} |h(\eta)|^2 |d\eta| = \|h\|_{H^2}^2.
\end{align*}

The inverse identity $W_\varphi W_{\varphi^{-1}}=W_{\varphi^{-1}}W_\varphi=I$ is easy to check and so we leave it to the reader.
\end{proof}
\begin{theorem}\label{T:biPure}
Let $(V_1, V_2)$ be a bi-pure pair of isometries acting on a Hilbert space $\mathcal{H}$ satisfying
$$
V_1 V_2 = V_2 \varphi(V_1),
$$
where $\varphi \in \mathrm{Aut}(\bD)$. Then $(V_1, V_2)$ is jointly unitarily equivalent to 
$$
(M_z^{\cD_{V_1^*}}, M_\Theta W_\varphi^*)
$$ acting on $H^2(\cD_{V_1^*})$, where $W_\varphi$ is the unitary weighted composition operator on $H^2$ as in Lemma \ref{L:WeightCompos} and $\Theta \in H^\infty(\bD, \mathcal{B}(\cD_{V_1^*}))$ is an operator-valued inner function such that $M_\Theta$ is a shift.
\end{theorem}
\begin{proof}
By the Wold decomposition, $V_1$ is unitarily equivalent to $M_z \otimes I_{\mathcal{E}}$ on $H^2(\cD_{V_1^*})$. Under this unitary identification, $\varphi(V_1)$ acts as $M_\varphi \otimes I_{\mathcal{D}_{V_1^*}}$. Direct evaluation shows that $W_\varphi M_z = M_\varphi W_\varphi$, or equivalently, $M_\varphi = W_\varphi M_z W_\varphi^*$. Substituting $M_\varphi$ into the covariance identity $M_z V_2 = V_2 M_\varphi$ yields:
$$
M_z V_2 = V_2 (W_\varphi M_z W_\varphi^*) \implies M_z (V_2 W_\varphi) = (V_2 W_\varphi) M_z.
$$
Hence, the isometry $A := V_2 W_\varphi$ commutes with $M_z \otimes I_{\mathcal{D}_{V_1^*}}$. It is well known that any bounded operator on $H^2(\bD, \mathcal{E})$ commuting with $M_z \otimes I_{\mathcal{E}}$ is multiplication by an operator-valued analytic function $\Theta \in H^\infty(\bD, \mathcal{B}(\mathcal{E}))$; moreover, when the bounded operator is an isometry, the function $\Theta$ must be inner. Consequently, $V_2 W_\varphi = M_\Theta$ for some inner $\Theta$, which implies $V_2 = M_\Theta W_\varphi^*.$

It remains to check that $V_2W_\varphi = M_\Theta$ is a pure isometry. If $M_\Theta$ is not pure, then it is well known (see e.g., \cite[Section V.2]{NF_book}) that there is a nonzero subspace $\mathcal L \subset \cF$ reducing under $\Theta(z)$ for each $z$ such that $\Theta(z)|_{\mathcal L}$ is a unitary that is independent of $z$. Since $W_\varphi^*=W_{\varphi^{-1}}$ has no role on the coefficient space, it commutes with $M_\Theta$ and $W_{\varphi^{-1}}H^2(\mathcal L) = H^2(\mathcal L)$. We use these observations to note that
$$
V_2(H^2(\mathcal L)) = M_\Theta W_{\varphi^{-1}} H^2(\mathcal L) = M_\Theta  H^2(\mathcal L) = H^2(\mathcal L).
$$This contradicts the assumption that $V_2$ is a pure isometry.
\end{proof}
\begin{remark}
It is instructive to deduce Theorem \ref{T:biPure} as a special case of the general model \eqref{ModelPair}.
So let $(V_1, V_2)$ be represented as in \eqref{ModelPair} in the general functional model space as in \eqref{ModelSpace}. Since $V_1$ is a pure isometry, the unitary part in its Wold decomposition vanishes, so
$$
\{0\}=\cQ_{V_1^*} \cong H^2(\overline{\Delta_{\Theta, \varphi} L^2(\cF)}) \oplus H^2(\cE_0) \oplus \cQ_{E^*}.
$$
This forces each component space to be trivial:
$$
\cQ_{E^*} = \{0\}, \quad \cE_0 = \{0\}, \quad \text{and} \quad H^2(\overline{\Delta_{\Theta, \varphi} L^2(\cF)}) = \{0\}.
$$
In particular, $\overline{\Delta_{\Theta, \varphi} L^2(\cF)} = \{0\}$, which implies that the weighted defect operator vanishes almost everywhere on the unit circle:
$$
\Delta_{\Theta, \varphi}(\zeta) = \left( I_\cF - |\varphi'(\zeta)| \Theta(\varphi(\zeta))^* \Theta(\varphi(\zeta)) \right)^{\frac{1}{2}} = 0 \quad \text{a.e. } \zeta \in \bT.
$$
This yields
\begin{equation} \label{defect_zero}
|\varphi'(\zeta)| \Theta(\varphi(\zeta))^* \Theta(\varphi(\zeta)) = I_\cF \quad \text{a.e. } \zeta \in \bT.
\end{equation}
Making the change of variable $\eta = \varphi(\zeta) \in \bT$ (so $\zeta = \varphi^{-1}(\eta)$) and utilizing the derivative identity $|\varphi'(\varphi^{-1}(\eta))| = \frac{1}{|(\varphi^{-1})'(\eta)|}$, equation \eqref{defect_zero} becomes
\begin{equation} \label{theta_norm}
\Theta(\eta)^* \Theta(\eta) = |(\varphi^{-1})'(\eta)| I_\cF \quad \text{a.e. } \eta \in \bT.
\end{equation}
This shows that $\Theta$ is not inner and indicates that a scaling is possible to obtain one. Before we get hold of the inner function, let us note that since $V_1$ is a shift, the general model space \eqref{ModelSpace} collapses to $\cK = H^2(\cF)$, and the model operator pair simplifies to
$$
V_1 = M_z^\cF = M_z \otimes I_\cF \quad \text{and} \quad V_2 = M_\Theta C_{\varphi^{-1}}.
$$
Next, recall from Lemma \ref{L:WeightCompos} that $W_\varphi h(z) = \sqrt{\varphi'(z)} h(\varphi(z))$ is unitary on $H^2$ and
$$
(W_\varphi^* h)(z) = \sqrt{(\varphi^{-1})'(z)} \, h(\varphi^{-1}(z)).
$$
Consequently, the composition operator $C_{\varphi^{-1}}$ factors as $C_{\varphi^{-1}} = M_{\frac{1}{\sqrt{(\varphi^{-1})'}}} W_\varphi^*$. We can therefore rewrite $V_2$ as
$$
V_2 = M_\Theta C_{\varphi^{-1}} = M_\Theta M_{\frac{1}{\sqrt{(\varphi^{-1})'}}} W_\varphi^* = M_{\widetilde{\Theta}} W_\varphi^*,
$$
where $\widetilde{\Theta}(z) := \frac{\Theta(z)}{\sqrt{(\varphi^{-1})'(z)}} = \Theta(z) \sqrt{\varphi'(\varphi^{-1}(z))}$.

Since $\Theta \in H^\infty(\bD, \mathcal{B}(\cF))$ and $(\varphi^{-1})'$ is analytic, bounded, and non-vanishing on $\bD$, the function $\widetilde{\Theta}$ belongs to $H^\infty(\bD, \mathcal{B}(\cF))$. Furthermore, using \eqref{theta_norm}, its boundary values satisfy
$$
\widetilde{\Theta}(\zeta)^* \tilde{\Theta}(\zeta) = \frac{\Theta(\zeta)^* \Theta(\zeta)}{|(\varphi^{-1})'(\zeta)|} = \frac{|(\varphi^{-1})'(\zeta)| I_\cF}{|(\varphi^{-1})'(\zeta)|} = I_\cF \quad \text{a.e. } \zeta \in \bT.
$$
Thus $\widetilde{\Theta}$ is an operator-valued inner function on $\bD$. The fact that $\widetilde\Theta$ must make the multiplication operator $M_{\widetilde\Theta}$ a shift operator follows from the discussion in the last paragraph of the proof of Theorem \ref{T:biPure}.
\end{remark}

\subsection{The parameter set as a complete unitary invariant}
In this section we show that the parameter set as described in \eqref{Parameter} is a complete set of unitary invariant as stated in the theorem below.
\begin{theorem} \label{T:Coin} 
Let $(V_1,V_2)$ on $\cK$ and $(\tV_1,\tV_2)$ on $\tcK$ be two isometric pairs satisfying the automorphic covariance relations
$$
V_1 V_2 = V_2 \varphi(V_1) \quad \text{and} \quad \tV_1 \tV_2 = \tV_2 \varphi(\tV_1).
$$
Let $(V_1, V_2)$ and $(\tV_1, \tV_2)$ be unitarily equivalent to the functional model operator pair as in \eqref{ModelPair} associated with the parameter tuples $\Xi = (\cF, \cE_0, \Theta, u_{\cE_0}, Y_1, Y_2)$ and $\tXi = (\tcF, \tcE_0, \tTheta, u_{\tcE_0}, \tY_1, \tY_2)$, respectively. Then the following are equivalent:
\begin{enumerate}
\item $(V_1, V_2)$ and $(\tV_1, \tV_2)$ are jointly unitarily equivalent.

\item The associated parameter tuples $\Xi$ and $\tXi$ coincide, i.e., there exist unitary operators
$$
\tau_\cF \colon \cF \to \tcF, \quad \tau_{\cE_0} \colon \cE_0 \to \tcE_0, \quad \tau_u \colon \cK_u \to \tcK_u
$$
so that
\begin{align}
& \tau_\cF \Theta(z) = \tTheta(z) \tau_\cF \text{ for all } z \in \mathbb{D}, \notag \\
& \tau_{\cE_0} u_{\cE_0} = u_{\tcE_0} \tau_{\cE_0}, \notag \\
& \tau_u (Y_1, Y_2) = (\tY_1, \tY_2) \tau_u. \label{Equiv}
\end{align}
\end{enumerate}
\end{theorem}

\begin{proof}[Proof of (ii) $\Rightarrow$ (i):]
 Since $\tau_\cF \Theta(z) = \tTheta(z) \tau_\cF$, it follows that $\tau_\cF \Theta(\zeta)^* \Theta(\zeta) = \tTheta(\zeta)^* \tTheta(\zeta) \tau_\cF$ for almost every $\zeta \in \mathbb{T}$. Hence $\tau_\cF$ intertwines the weighted defect operators pointwise on $\mathbb{T}$:
$$
\tau_\cF \Delta_{\Theta, \varphi}(\zeta) = \Delta_{\tTheta, \varphi}(\zeta) \tau_\cF \quad \text{a.e. } \zeta \in \mathbb{T}.
$$
The unitary $\tau_\cF$ gives rise to another unitary operator $\widehat{\tau_\cF} \colon \overline{\Delta_{\Theta, \varphi} L^2(\cF)} \to \overline{\Delta_{\tTheta, \varphi} L^2(\tcF)}$ defined on the dense set $\{\varphi^{-n} \Delta_{\Theta, \varphi} f : n \geq 0, f \in H^2(\cF)\}$ by
\begin{align}\label{WideHatTau}
\widehat{\tau_\cF} \left( \varphi^{-n} \Delta_{\Theta, \varphi} f \right) = \varphi^{-n} \Delta_{\tTheta, \varphi} \tau_\cF f.
\end{align}
Note that $\widehat{\tau_\cF}$ satisfies the intertwining property
$$
\widehat{\tau_\cF} M_\varphi|_{\overline{\Delta_{\Theta, \varphi} L^2(\cF)}} = M_\varphi|_{\overline{\Delta_{\tTheta, \varphi} L^2(\tcF)}} \widehat{\tau_\cF}.
$$
Using \eqref{WideHatTau}, a direct matrix calculation shows that the diagonal unitary operator
\begin{align}\label{Tau_General}
\begin{bmatrix}
I_{H^2} \otimes \tau_\cF & & & \\
& I_{H^2} \otimes \widehat{\tau_\cF} & & \\
& & I_{H^2} \otimes \tau_{\cE_0} & \\
& & & \tau_u
\end{bmatrix} \colon
\begin{bmatrix}
H^2(\cF) \\
H^2(\overline{\Delta_{\Theta, \varphi} L^2(\cF)}) \\
H^2(\cE_0) \\
\cK_u
\end{bmatrix} \to
\begin{bmatrix}
H^2(\tcF) \\
H^2(\overline{\Delta_{\tTheta, \varphi} L^2(\tcF)}) \\
H^2(\tcE_0) \\ 
\tcK_u
\end{bmatrix}
\end{align}
identifies the functional models of $(V_1, V_2)$ and $(\tV_1, \tV_2)$.

\medskip

\noindent
{\textit{Proof of (i) $\Rightarrow$ (ii):}} Assume that 
$$
\tau \colon \begin{bmatrix} H^2(\cF) \\ H^2(\overline{\Delta_{\Theta, \varphi} L^2(\cF)}) \\ H^2(\cE_0) \\ \cK_u \end{bmatrix} 
\to
 \begin{bmatrix} H^2(\tcF) \\ H^2(\overline{\Delta_{\tTheta, \varphi} L^2(\tcF)}) \\ H^2(\tcE_0) \\ \tcK_u \end{bmatrix}
$$
is unitary and $\tau (V_1, V_2) = (\tV_1, \tV_2) \tau$, where $(V_1, V_2)$ and $(\tV_1, \tV_2)$ are in the model form \eqref{ModelPair} corresponding to parameter tuples $\Xi$ and $\tXi$, respectively.

The operators $V_1$ and $\tV_1$ have the forms
$$
\begin{bmatrix} M_z^\cF & 0 \\ 0 & W \end{bmatrix} \quad \text{and} \quad
\begin{bmatrix} M_z^\tcF & 0 \\ 0 & \tW \end{bmatrix}
$$
where $W$ and $\tW$ are unitaries. Since $\tau$ intertwines $V_1$ and $\tV_1$, Part (2) of Lemma \ref{L:AuxLemma} implies that $\tau$ must be of the form
\begin{align}\label{Part2Lem}
\tau = \begin{bmatrix} I_{H^2} \otimes \tau_\cF & 0 \\ 0 & \sigma \end{bmatrix}
\end{align}
for some unitary operator $\tau_\cF \colon \cF \to \tcF$ and
$$
\sigma \colon \begin{bmatrix} H^2(\overline{\Delta_{\Theta, \varphi} L^2(\cF)}) \\ H^2(\cE_0) \\ \cK_u \end{bmatrix}
\to \begin{bmatrix} H^2(\overline{\Delta_{\tTheta, \varphi} L^2(\tcF)}) \\ H^2(\tcE_0) \\ \tcK_u \end{bmatrix}
$$
is a unitary such that it intertwines
\begin{align}\label{TheWs}
W_1 = \begin{bmatrix} I_{H^2} \otimes_\varphi M_\varphi & 0 & 0 \\ 0 & I_{H^2} \otimes_\varphi u_{\cE_0} & 0 \\ 0 & 0 & Y_1 \end{bmatrix} \quad \text{and} \quad
\tW_1 = \begin{bmatrix} I_{H^2} \otimes_\varphi M_\varphi & 0 & 0 \\ 0 & I_{H^2} \otimes_\varphi u_{\tcE_0} & 0 \\ 0 & 0 & \tY_1 \end{bmatrix}.
\end{align}
Since $\tau$ also intertwines $V_2$ and $\tV_2$, comparing the (1,1) block gives
\begin{align*}
(I_{H^2}\otimes \tau_\cF)M_{\Theta}C_{\varphi^{-1}} = M_{\Theta'}C_{\varphi^{-1}}(I_{H^2}\otimes \tau_\cF),
\end{align*}which, when evaluated on the $\cF$-valued constant functions in $H^2(\cF)$, gives
\begin{align}\label{ConstantEvaluation}
\tau_\cF \Theta(z)= \Theta'(z)\tau_\cF.
\end{align}As noted in the proof of (ii)$\Rightarrow$(i), this readily gives  a unitary operator $\widehat{\tau_\cF}: \overline{\Delta_{\Theta,\varphi}L^2(\cF)}\to \overline{\Delta_{\widetilde\Theta,\varphi}L^2(\widetilde\cF)}$ such that
$$
\widehat{\tau_\cF} \Delta_{\Theta, \varphi}|_{H^2(\cF)} = (\Delta_{\tTheta, \varphi}|_{H^2(\tcF)}) \tau_\cF.
$$Comparison of the (1,2) block of $\tau V_2=\widetilde V_2\tau$ gives
\begin{align}\label{12Block}
\sigma \begin{bmatrix} \Delta_{\Theta, \varphi} \\ 0 \\ 0 \end{bmatrix} = \begin{bmatrix} \Delta_{\tTheta, \varphi} \\ 0 \\ 0 \end{bmatrix} \tau_\cF
\end{align}
while the comparison of the (2,2) block gives
\begin{align*}
\sigma \begin{bmatrix} M_z^\cG & 0 \\ 0 & Y_2 \end{bmatrix} = \begin{bmatrix} M_z^\tcG & 0 \\ 0 & \tY_2 \end{bmatrix} \sigma,
\end{align*}
where $\mathcal G = \begin{bmatrix} \overline{\Delta_{\Theta, \varphi} L^2(\cF)} \\ \cE_0 \end{bmatrix}$ and $\tcG = \begin{bmatrix} \overline{\Delta_{\tTheta, \varphi} L^2(\tcF)} \\ \tcE_0 \end{bmatrix}$.

Applying Part (2) of Lemma \ref{L:AuxLemma} again to $\sigma$ implies that $\sigma$ is of the form
\begin{align}\label{TheSigma}
\sigma = \begin{bmatrix} I_{H^2} \otimes \sigma_\cG & 0 \\ 0 & \tau_u \end{bmatrix}
\end{align}
for some unitary operators $\sigma_\cG \colon \cG \to \tcG$ and $\tau_u \colon \cK_u \to \tcK_u$ satisfying
\begin{align}\label{Part2Again}
(I_{H^2} \otimes \sigma_\cG) M_z^\cG = M_z^\tcG (I_{H^2} \otimes \sigma_\cG) \quad \text{and} \quad \tau_u Y_2 = \tY_2 \tau_u.
\end{align}
Representing $\sigma_\cG \colon \cG \to \tcG$ as the $2 \times 2$ matrix
$$
\sigma_\cG = \begin{bmatrix} \sigma_{\cG, 11} & \sigma_{\cG, 12} \\ \sigma_{\cG, 21} & \tau_{\cE_0} \end{bmatrix} \colon \begin{bmatrix} \overline{\Delta_{\Theta, \varphi} L^2(\cF)} \\ \cE_0 \end{bmatrix} \to \begin{bmatrix} \overline{\Delta_{\tTheta, \varphi} L^2(\tcF)} \\ \tcE_0 \end{bmatrix},
$$
it follows from \eqref{12Block} that
\begin{align}\label{12Block'}
\sigma_{\cG, 11} \Delta_{\Theta, \varphi}|_{H^2(\cF)} = (\Delta_{\tTheta, \varphi}|_{H^2(\tcF)}) \tau_\cF.
\end{align}
Since $\sigma$ intertwines $W_1$ and $\tW_1$ as in \eqref{TheWs}, restriction to the coefficient space $\cG$ yields
\begin{align}\label{CoeffSigma}
\sigma_{\cG, 11} M_\varphi = M_\varphi \sigma_{\cG, 11}, \quad \tau_{\cE_0} u_{\cE_0} = u_{\tcE_0} \tau_{\cE_0}, \quad \text{and} \quad \tau_u Y_1 = \tY_1 \tau_u.
\end{align}
Combining $\sigma_{\cG, 11} M_\varphi = M_\varphi \sigma_{\cG, 11}$ with \eqref{12Block'}, we calculate for every $n \geq 0$ and $f \in H^2(\cF)$:
\begin{align*}
\sigma_{\cG, 11} \left( \varphi^{-n} \Delta_{\Theta, \varphi} f \right) &= \varphi^{-n} \sigma_{\cG, 11} \left( \Delta_{\Theta, \varphi} f \right) \\
&= \varphi^{-n} \Delta_{\tTheta, \varphi} \tau_\cF f = \widehat{\tau_\cF} \left( \varphi^{-n} \Delta_{\Theta, \varphi} f \right).
\end{align*}
Since $\{\varphi^{-n} \Delta_{\Theta, \varphi} f : n \geq 0, f \in H^2(\cF)\}$ is dense in $\overline{\Delta_{\Theta, \varphi} L^2(\cF)}$, we conclude that
$$
\sigma_{\cG, 11} = \widehat{\tau_\cF}.
$$
Because $\sigma_\cG$ is unitary and its $(1,1)$-entry $\sigma_{\cG, 11} = \widehat{\tau_\cF}$ is itself unitary, $\sigma_\cG$ must be block-diagonal:
$$
\sigma_\cG = \begin{bmatrix} \widehat{\tau_\cF} & 0 \\ 0 & \tau_{\cE_0} \end{bmatrix} \colon \begin{bmatrix} \overline{\Delta_{\Theta, \varphi} L^2(\cF)} \\ \cE_0 \end{bmatrix} \to \begin{bmatrix} \overline{\Delta_{\tTheta, \varphi} L^2(\tcF)} \\ \tcE_0 \end{bmatrix}.
$$
Thus $\tau$ is precisely of the form \eqref{Tau_General}. Finally, \eqref{ConstantEvaluation}, \eqref{Part2Again}, and \eqref{CoeffSigma} confirm that $\tau_\cF, \tau_{\cE_0}, \tau_u$ satisfy all coincidence conditions in \eqref{Equiv}.
\end{proof}
\section{Concluding Remarks}
In this paper, we restricted our attention to disk automorphisms. Several intermediate structural steps rely heavily on the bijectivity and invertibility of automorphisms—properties that present substantial new challenges when considering more general maps, such as finite Blaschke products. For instance:
\begin{enumerate}
    \item When $\varphi$ is an automorphism, the right covariance relation $V_1 V_2 = V_2 \varphi(V_1)$ is equivalent to the left covariance relation $\varphi^{-1}(V_1) V_2 = V_2 V_1$. This equivalence generally fails when $\varphi$ is not an automorphism.
    \item Lemma~\ref{L:WeightCompos} does not seem to hold for general finite Blaschke products.
\end{enumerate}
Both Lemma~\ref{L:WeightCompos} and the duality between left and right covariance relations were crucial in establishing the model in \S \ref{SS:FunctModel} -- see item~(M1).

In a subsequent paper, we intend to address these obstacles and extend the functional model construction to isometric pairs governed by general finite Blaschke products. Furthermore, we plan to develop an analogue of the classical Berger--Coburn--Lebow functional model tailored to this dynamically twisted framework. We note that a Berger--Coburn--Lebow-type model was recently established in \cite{BallSau-qComm} for the special case of $q$-commuting isometric pairs corresponding to $\varphi(z) = qz$ with $|q| = 1$.

\end{document}